\documentclass[11pt,twoside]{amsart}
\usepackage{amsmath, amsthm, amsfonts, amssymb}
\usepackage[numbers, sort & compress]{natbib}
\usepackage[mathcal,mathscr]{eucal}
\usepackage{mathrsfs}
\usepackage{graphicx, xcolor}
\usepackage[colorlinks]{hyperref}

\newtheorem{theorem}{Theorem}[section]

\newtheorem{corollary}[theorem]{Corollary}
\theoremstyle{definition}

\newtheorem{example}[theorem]{Example}

\newtheorem{remark}[theorem]{Remark}

\begin{document}

\title[\scriptsize{Cartesian Symmetry Classes}]{Cartesian Symmetry Classes associated with certain subgroups of $S_m$}

\author[Seyyed Sadegh Gholami and Yousef Zamani]{Seyyed Sadegh Gholami and Yousef Zamani$^{\ast }$}

\address{Department of Mathematics, Faculty of Basic Sciences, Sahand University of Technology, Tabriz, Iran}
\email{s\_gholami98@sut.ac.ir}

\address{Department of Mathematics, Faculty of Basic Sciences, Sahand University of Technology, Tabriz, Iran}
\email{zamani@sut.ac.ir}
\thanks{$^{\ast}$ Corresponding author}

\maketitle
\begin{abstract}
Let $V$ be an $n$-dimensional inner product space. Assume $G$ is a subgroup of the symmetric group of degree $m$, and $\lambda$ is an irreducible character of $G$. Consider the \emph{Cartesian symmetrizer} $C_{\lambda}$ on the Cartesian space $\times^{m}V$ defined by
\[
C_{\lambda} = \frac{\lambda(1)}{|G|}\sum_{\tau\in G} \lambda(\tau) Q(\tau).
\]
The vector space
$
V^{\lambda}(G) = C_{\lambda}(\times^{m}V)
$
is called the Cartesian symmetry class associated with $G$ and $\lambda$.

In this paper, we give a formula for the dimension of the cyclic subspace $V^{\lambda}_{ij}$. Then we discuss the problem existing an $O$-basis for the Cartesian symmetry class $V^{\lambda}(G)$. Also, we compute the dimension of the symmetry class $V^{\lambda}(G)$ when $G = \langle \sigma_{1} \sigma_{2} \cdots \sigma_{p} \rangle$ or $G = <\sigma_{1}><\sigma_{2}> \cdots <\sigma_{k}>$, where $\sigma_i$ are disjoint cycles in $S_{m}$. The dimensions are expressed in terms of the Ramanujan sum.

Additionally, we provide a necessary and sufficient condition for the existence of an $O$-basis  for Cartesian symmetry classes associated with the irreducible characters of the dihedral group $D_{2m}$. The dimensions of these classes are also computed.
\end{abstract}
\hspace{-1cm}\rule{\textwidth}{0.2mm}
\section{\bf Introduction and Preliminaries}
The Cartesian symmetry class associated with a subgroup $G$ of the symmetric group $S_m$ and a mapping $f:G\rightarrow \mathbb{C}$ was initially explored by Tian-Gang Lei in \cite{Lei}. The dimension of this Cartesian symmetry class, corresponding to an irreducible character of the subgroup $G$ in $S_{m}$, has been computed in terms of the fixed point character of $S_m$ in \cite{Zamani 2}. Additionally, a formula expressing the dimension of the Cartesian symmetry class in terms of the rank of an idempotent matrix. The article also deduces certain properties of generalized trace functions of a matrix.

In this research, we give a formula for the dimension of the cyclic subspace $V^{\lambda}_{ij}$. Then we discuss the problem existing an $O$-basis for the Cartesian symmetry class $V^{\lambda}(G)$. Also, we compute the dimension of the symmetry class $V^{\lambda}(G)$ in terms of the Ramanujan sum when $G = \langle \sigma_{1} \sigma_{2} \cdots \sigma_{p} \rangle$ or $G = <\sigma_{1}><\sigma_{2}> \cdots <\sigma_{k}>$, where $\sigma_i$ are disjoint cycles in $S_{m}$. Additionally, we provide a necessary and sufficient condition for the existence of an $O$-basis  for Cartesian symmetry classes associated with the irreducible characters of the Dihedral group $D_{2m}$. The dimensions of these classes are also computed.

We commence with a review of Cartesian symmetry classes. For more details, we refer the reader to \cite{Lei,Zamani 2}.
 
Let $V$ be a complex inner product space of dimension $n \geq 2$. Consider $\mathbb{E}=\{ f_{1}, \ldots , f_{n}\}$ as an orthonormal basis of $V$. Let  $\times^{m}V$ be the Cartesian product of $m$-copies of $V$.
We have an induced inner product of $\times^{m}V$, which is defined by
$$
< z^{\times} , w^{\times}> = \sum_{i=1}^{m}<z_{i} , w_{i}>,
$$
where 
$$
z^{\times}=(z_{1},\ldots,z_{ m}),\  w^{\times}=(w_{1},\ldots,w_{m}).
$$

For every $1\leq i \leq n ,\ 1\leq j \leq m$, we define
 $$
f_{ij}=(\delta_{1j}f_{i},\delta_{2j}f_{i},\ldots , \delta_{mj}f_{i}) \in \times^{m}V.
$$
Then  
$$
\mathbb{E}^{\times}=\lbrace f_{ij}~|~ i=1,2,\ldots ,n , j=1,2,\ldots , m \rbrace
$$ 
is an orthonormal basis of  $\times^{m}V$.\\

Let $G$ be a subgroup of the symmetric group $S_m$. For any $\tau \in G$, we define the Cartesian permutation linear operator
 \begin{equation*}
   Q(\tau):\times^{m}V\longrightarrow \times^{m}V
\end{equation*}
by
\begin{equation*}
 Q(\tau)(u_1,\ldots , u_m )=(u_{\tau^{-1}(1)},\ldots, u_{\tau^{-1}(m)}).
 \end{equation*}
It is easy to see that $Q:\tau \rightarrow Q(\tau)$ defines a faithful unitary representation of $G$ over 
$\times^{m}V$.\\

Let $\lambda$ be a complex irreducible character of $G$. Define the \emph{Cartesian symmetrizer} $C_{\lambda}$ as follows:
\[
C_{\lambda }=\dfrac{\lambda(1)}{|G|}\sum_{\tau\in G} \lambda(\tau) Q(\tau).
\]
It has been shown \cite{Zamani 2} that $C_{\lambda}$ is an orthogonal projection on $\times^{m}V$, i.e., $C_{\lambda}^2= C_{\lambda}^\ast=  C_{\lambda}$.

The range of the linear operator $C_{\lambda}$ is called \emph{the Cartesian symmetry class associated with $G$ and $\lambda$}, denoted by $V^{\lambda}(G)$. It is easy to see that 
\[
\times^{m}V=\sum_{\lambda \in Irr(G)}^{\bot} {V^{\lambda}(G)}~ (\text{orthogonal direct sum}).
\]
The elements 
\[
u^{\lambda}=C_{\lambda}(u_{1},\ldots,u_{m})
\]
of $V^{\lambda}(G)$ are called \emph{$\lambda$-symmetrized vectors}. Note that $V^{\lambda}(G)$ is spanned by the standard $\lambda$-symmetrized vectors 
\[
f^{\lambda}_{ij}=C_{ \lambda}(f_{ij})~(i=1,2,\ldots, n, j=1,2,\ldots, m).
\]

\begin{theorem}\cite{Zamani 2}\label{p1} 
Let $G$ be a subgroup of the symmetric group $S_m$ and $\lambda \in \text{Irr}(G)$. Then for each $1\leq i,r\leq n, 1\leq j,s\leq m$, we have
\begin{eqnarray*}
\langle f_{ij}^{\lambda} ,f_{rs}^{\lambda}\rangle&=&
\dfrac{\lambda(1)^2}{|G|^2}\sum_{\sigma , \tau\in G} \lambda(\sigma) {\overline{\lambda(\tau)}}{\delta_{ir}}\delta_{\sigma^{-1}(j),\tau^{-1}(s)}\\
&=& \begin{cases}
   0    &\text{if $i\neq r$ or $j\nsim s$ }\\
     \frac{\lambda (1)}{|G|}\sum_{g\in G_{j}}\lambda(g\tau^{-1})             &\text{if $s=\tau(j)$}
  \end{cases}.
\end{eqnarray*}
In particular
\begin{eqnarray}\label{Norm}
 \parallel f_{ij}^{\lambda}\parallel^2=\dfrac{\lambda(1)[\lambda,1_{G_j}]}{{[G:G_j]}},
\end{eqnarray}
 where $G_{j}$ is the stabilizer of $j$ in $G$  and $[\ ,\ ]$ the inner product of characters \cite{Isaacs}.
\end{theorem}

Let $\mathcal{D}$ be a set of representatives of orbits of the action of $G$ on the set $\mathbf{I}_m=\{ 1,2,\cdots, m \}$ and
$
\mathcal{O}=\lbrace j ~|~ 1\leq j\leq m,\ [\lambda , 1_{G_j}]\neq 0 \rbrace.
$
We put
$\bar {\mathcal{D}} =\mathcal{D} \cap \mathcal{O}$.
It is easy to see that
the set
 $$
\lbrace f^{\lambda}_{ij}~\mid ~ 1\leq i \leq n , j \in \bar D \rbrace
 $$
 is an orthogonal set of non-zero vectors in $V^\lambda (G)$. Notice that the above set may not be a basis for $V^\lambda (G)$.\\
For any $ 1\leq i \leq n$ and  $ j \in \bar{\mathcal{D}}$, define the cyclic subspace
$$
V^{\lambda}_{ij} = \langle f^{\lambda}_{i \sigma( j)}\mid  \sigma \in G \rangle.
$$
Clearly
\begin{eqnarray}\label{ODS}
V^\lambda (G) =\sum_{1\leq i\leq n,~j\in \bar{\mathcal{D}}}^{\bot} {V^{\lambda}_{ij}}
\end{eqnarray}
the orthogonal direct sum of the cyclic subspaces $V^{\lambda}_{ij}\ (1\leq i\leq n,\ j\in \bar{\mathcal{D}})$. If $\lambda$ is a linear character of $G$ and 
$ j \in \bar{\mathcal{D}}$, then it is easy to see that $f^{\lambda}_{i \sigma( j)}=\lambda(\sigma^{-1})f^{\lambda}_{i j}$, so $\dim V^{\lambda}_{ij}=1$ and the set 
$$
\lbrace f^{\lambda}_{ij}\mid 1\leq i \leq n , j \in \bar{\mathcal{D}} \rbrace
 $$
is an orthogonal basis of  $V^\lambda (G)$. Therefore, we have  $\dim V^{\lambda}(G) = (\dim V) |\bar{ \mathcal{D}}|$. For example we have
$$
\dim V^{1}(S_{m})=\dim V,\ \ \ 
  \dim V^{\varepsilon}(S_{m}) = 
  \begin{cases}
   \dim V    &\text{if $m=2$}\\
     0             &\text{if $m \geq 3$}
  \end{cases},
$$
where $1$ and $\varepsilon$ are the principal and the alternating character of $S_m$, respectively.

Suppose $\lambda$ is a non-linear irreducible character of $G$. We will now proceed to construct a basis for $V^{\lambda}(G)$. For each $j\in\bar{\mathcal{D}}$, we choose the set $\{ j_{1}, \ldots, j_{s_{j}} \}$ from the orbit of $j$ such that
$
\{f^{\lambda}_{ij_{1}},  \ldots , f^{\lambda}_{ij_{s_j}} \}
$
forms a basis for the cyclic subspace $V^{\lambda}_{ij}$. This procedure is executed for each $k\in\bar{\mathcal{D}}$. If $\bar{\mathcal{D}} =\{j, k, l, \ldots \}$ with $j<k<l< \ldots$, then we define
$$
\hat{\mathcal{D}}=\{j_{1}, \ldots, j_{s_{j}}; k_{1},\ldots, k_{s_{k}}, \ldots\}
$$
with an indicated order. Consequently,
 $$
\mathbb{E}^{\lambda}=\{ f^{\lambda}_{ij}~|~1\leq i\leq n,\  j\in \hat{\mathcal{D}}\}
$$
forms a basis for $V^\lambda (G)$. Note that this basis may not be orthogonal. Therefore,
$$
\dim V^\lambda (G)=(\dim V)|\hat{\mathcal{D}}|.
$$
If the subspace $W$ of $V^{\lambda}(G)$ possesses a basis consisting of orthogonal standard $\lambda$-symmetrized vectors, we will refer to $W$ as having an $O$-basis.

In section 2 we give a formula for the dimension of the cyclic subspace $V^{\lambda}_{ij}$. Then we discuss the problem existing  $O$-basis for $V^{\lambda}(G)$. In section 3, we determine the dimensions of Cartesian symmetry classes associated with a cyclic subgroup of the symmetric group $S_m$ (generated by a product of disjoint cycles) and the product of cyclic subgroups of $S_m$. These dimensions are expressed in terms of the Ramanujan sum. \\
In Section 4, we establish a necessary and sufficient condition for the existence of an $O$-basis for Cartesian symmetry classes associated with the irreducible characters of the dihedral group. Additionally, we compute the dimensions of these classes.
\section{\bf Orthogonal bases of Cartesian symmetry classes}
In this section we discuss the problem existing $O$-basis for $V^{\lambda}(G)$ consisting of standard $\lambda$-symmetrized vectors. We first prove the following theorem.
\begin{theorem}\label{cyclic subspace}
For every $1\leq i\leq n$ and $j\in \bar{\mathcal{D}}$, we have
$$
\dim V^{\lambda}_{ij}=\lambda(1)[ \lambda , 1_{G_{j}}].
$$
 \end{theorem}
 \begin{proof}
Let
 $[G:G_j]=k$ , $G=\bigcup_{i=1}^{k} \tau_{i}G_{j}$
be the left coset decomposition of $G_{j}$ in $G$ and  
$$
Orb~(j) =\lbrace \tau_{1}(j), \ldots, \tau_{k}(j) \rbrace.
$$
We put
\begin{align*}
V_{ij}=&\langle  f_{i \sigma_(j)}\mid \sigma \in G \rangle.
\end{align*}
Then
\begin{align*}
\mathbb{E}_{ij}=\lbrace  f_{i \tau_1(j)}, \ldots,f_{i \tau_k(j)}\rbrace
\end{align*}
is a basis for $V_{ij}$ but the set 
\begin{align*}
\lbrace  f_{i \tau_{1}(j)}^\lambda ,\ldots, f_{i \tau_{k}(j)}^\lambda\rbrace,
\end{align*}
may not be a basis for $V_{ij}^{\lambda}$.
Notice that $V_{ij}^{\lambda}=C_{\lambda}(V_{ij})$. Since $V_{ij}$ is invariant under $C_{ \lambda}$, so
we put 
$$
C_{ij}(G, \lambda)=C_{ \lambda} \mid  _{V_{ij}}.
$$
This  restriction is a linear operator on $V_{ij}$.
Since $C_{ \lambda}$ is an orthogonal projection on $\times^{m}V$, so $ C_{ij}(G, \lambda)$ is also an orthogonal projection on $V_{ij}$. 

Let $M=(m_{ij})=[ C_{ij}(G, \lambda)]_{\mathbb{E}_{ij}}$. Then 
$$
\dim V_{ij}^{\lambda}=rank~ C_{ij}(G, \lambda)=tr~ C_{ij}(G, \lambda)=tr~M.
$$
Now we compute the entries of the matrix $M$. We have
 \begin{align*}
C_{ij}(G, \lambda) f_{i \tau_{q}(j)} =&C_{\lambda}( f_{i\tau_{q}(j)}) \\
=&\frac{\lambda(1)}{|G|}\sum_{\sigma \in G} {\lambda ( \sigma) Q(\sigma)(f_{i \tau_{q}(j)})}\\
=&\frac{\lambda(1)}{|G|}\sum_{\sigma \in G} {\lambda ( \sigma) f_{i\sigma \tau_{q}(j)}}\\
=&\frac{\lambda(1)}{|G|}\sum_{\pi \in G} {\lambda ( \pi\tau_q^{-1}) f_{i\pi(j)}}   \hspace{10mm}(\sigma\tau_q=\pi)\\
=&\frac{\lambda(1)}{|G|}\sum_{p=1}^k\sum_{\pi\in {G_j}}\lambda(\tau_{p}\pi \tau^{-1}_{q})f_{i\tau_{p}\pi (j)}    \hspace{10mm}\\
=&\frac{\lambda(1)}{|G|}\sum_{p=1}^k\sum_{\pi \in{G_j}}\lambda(\tau_{p}\pi \tau^{-1}_{q})f_{i\tau_{p}(j)} \\
=&\sum_{p=1}^k\left(\frac{\lambda(1)}{|G|}\sum_{\pi \in{G_j}}\lambda(\tau_{p}\pi \tau^{-1}_{q}) \right)f_{i\tau_{p}(j)}.
\end{align*}
Thus
$$
m_{pq}=\frac{\lambda(1)}{|G|}\sum_{\pi \in {G_j}}\lambda (\tau_{p}\pi \tau _{q}^{-1}).
$$
Hence
 \begin{align*}
 \dim V^{\lambda}_{ij}=tr M=\sum_{p=1}^{k} m_{pp}
=\sum_{p=1}^{k}\frac{\lambda(1)}{|G|}\sum_{\pi \in{G_j}}\lambda (\tau_{p}\pi \tau^{-1}_{p})
=\sum_{p=1}^k\frac{\lambda(1)}{|G|}\sum_{\pi \in{G_j}}\lambda (\pi )\\
=\frac{k}{|G|}\lambda(1) \sum_{\pi \in{G_j}}\lambda (\pi )
=\lambda (1) \frac{1}{|G_j|} \sum_{\pi \in{G_j}}\lambda (\pi )
=\lambda(1)[\lambda, 1_{G_j}],
\end{align*}
so the result holds.
 \end{proof}
By using Theorem \ref{cyclic subspace} and Equation (\ref{ODS}), we have the following corollary.
 \begin{corollary}\label{dimension}
 Let $G$ be a subgroup of $S_{m}$. Then for any $\lambda \in \text{Irr}(G)$, we have 
 $$
 \dim V^{\lambda}(G) = (\dim V) \lambda (1)\sum_{j\in \bar{\mathcal{D}}}[\lambda, 1_{G_{j}}].
 $$
 \end{corollary}

 Similar to the proof of  Theorem 1 in \cite{Shahryari}, we can prove the following important theorem.
\begin{theorem}\label{o-basis}
Let $\lambda $ be a non-linear irreducible character of $G$ and suppose there exists $j\in \bar{\mathcal{D}}$ such that
$$
\frac{\sqrt{2}}{2} < \parallel f_{ij}^{\lambda}\parallel < 1.
$$
Then $V^{\lambda}(G)$ has no $O$-basis.
\end{theorem}
We will show that $\frac{\sqrt{2}}{2}$ is the best lower bound for $\parallel f_{ij}^{\lambda}\parallel$ in Theorem \ref{o-basis}. Also, we will give a counter-example that the converse of this theorem is not true.

 Recall that $G$ is $2$-transitive if $m\geq 2$ and $G_{i}$ is transitive on $\mathbf{I}_m\setminus \{i\}$ for some $1\leq i\leq m$.
 Let $\theta (\sigma)=\mid\{i~|~1\leq i\leq m, \sigma (i)=i\}\mid$. By \cite[Corollary 5.17]{Isaacs}, $G$ is $2$-transitive 
 if and only if $\lambda =\theta-1\in \text{Irr}(G)$. Now we have the following corollary.
 \begin{corollary}
Let  $G \leq S_{m}~(m\geq 3)$ be $2$-transitive and  $\lambda =\theta-1$. Then $V^{\lambda}(G)$ has no  $O$-basis.
In particular, if $G = S_{m}\ (m\geq3)$ or $G = A_{m}\ (m\geq4)$, then $V^{\lambda}(G)$ does not have an $O$-basis.
\end{corollary}
\begin{proof}
Clearly $G_{m}=G \cap S_{m-1}$ so $[ \lambda , 1_{G_{m}}]=1$ (Burnside's Lemma) 
and $[G : G_{m}]=[S_{m} : S_{m-1}]=m$ . By Equation (\ref{Norm}), we have 
$$
 \parallel f_{im}^{\lambda}\parallel^2=\frac{m-1}{m},
$$
so the result holds by the previous theorem.
\end{proof}
We will say that $\times^{m}V$ has $O$-basis if for each irreducible character $\lambda$ of $G$, the Cartesian symmetry class $V^{\lambda}(G)$ has an $O$-basis.
We know that $\times^{m}V=\sum_{\lambda \in \text{Irr}(G)}^{\bot}{V^{\lambda}(G)}$ (orthogonal direct sum), so the following result follows immediately from the above corollary.
\begin{corollary}
Assume $m \geq 3$. If $G\leq S_m$ is 2-transitive, then $\times ^{m}V$ does not have an $O$-basis. 
\end{corollary}

\section{\bf Cyclic Cartesian symmetry classes}
The well-known Ramanujan sum is given by
$$
\mathit{c}_{m}(q) = \sum_{s=0,\\ (s, m)=1}^{m-1}\exp(\frac{2 \pi iqs}{m}),
$$
where $m$ is a positive integer, and $q$ is a non-negative integer (see \cite{Apostol}).
  \begin{theorem} 
Let $G$ be the cyclic subgroup of $S_m$,  generated by the $m$-cycle $(1~2~\cdots ~m)$. Then for any $\lambda\in \text{Irr}(G)$, we have
$$
\dim V^\lambda (G)=\dim V.
$$ 
\end{theorem}
\begin{proof}
By \cite[Corollary 2.6]{Isaacs}, a finite group $G$ is abelian if and only if every irreducible character of it is linear. Let $\lambda\in \text{Irr}(G)$. Then $\lambda (1)=1$. It is easy to see that the action of $G$ on $\mathbf{I}_m$ is transitive, so $|\mathcal{D}|=1$. Assume that $\mathcal{D}=\{1\}$. Clearly $G_1=\{1\}$ and $[\lambda, 1_{G_1}]=1$. Therefore $\dim V^\lambda (G)=\dim V$, by 
Corollary \ref{dimension}. 
\end{proof}
Now, let's consider the cyclic group $G = \langle \sigma_{1} \sigma_{2} \cdots \sigma_{p} \rangle$, where $\sigma_i \ (1 \leq i \leq p)$ are disjoint cycles in $S_{m}$. Assume that the order of $\sigma_{i} \ (1 \leq i \leq p)$ is $m_i$. The irreducible characters of $G$ are given by 
$$
\lambda_{q}\left((\sigma_{1}\cdots\sigma_{p})^s\right) = \exp\left(\frac{2 \pi i qs}{M}\right), \quad s=0, \ldots , M - 1,
$$
where $M$ denotes the least common multiple of the integers $m_1, \ldots, m_p$.\\

In the following theorem we obtain $ \dim V^{\lambda _q}(G)$ in terms of  Ramanujan sum.
  \begin{theorem}
 Let $G =\langle \sigma_{1} \sigma_{2}\cdots  \sigma_{p}\rangle$, where $\sigma _{i}\ (1 \leq i\leq p)$ are disjoint cycles in $S_{m}$ of
 orders $m_{1}, \cdots , m_{p}$, respectively. Assume $m=\sum_{i=1}^{p} m_{i}$ and let $\lambda_{q},~ 0 \leq q \leq M-1$, be an irreducible complex character of $G$.
Then
$$
  \dim V^{\lambda _q}(G)=  (\dim V)\sum_{i=1}^{p}\frac{1}{m_{i}^{\prime}}
 \sum_{d\mid m_{i}^{\prime}} \mathit{c}_{m_{i}^{\prime}/d}(q),
$$
where $m_{i}^{\prime}=M/m_{i}$.
In particular, if $G$ is generated by an $m$-cycle in $S_{m}$, then $\dim V^{\lambda_{q}}(G)=\dim V$.
\end{theorem}
\begin{proof}
To compute the dimension $ \dim V^{\lambda_{q}}(G)$  we use Corollary \ref{dimension}. For every $1\leq i\leq p$, suppose 
$$
\mathcal{O}_i=\{ 1\leq j\leq m~\mid ~\sigma_i(j)\neq j\} .
$$
 Then 
 $\mathcal{O}_1,\cdots,\mathcal{O}_p$ are distinct orbits of $G$.
If $j\in \mathcal{O}_{i}$ then
 \begin{align*}
       G_j=& \{(\sigma_{1} \sigma_{2} \cdots \sigma_{p})^s ~|~ 0 \leq s \leq{M-1}, \sigma_{i}^{s}(j)=j \} \\
  =&\{\sigma_{1}^s \sigma_{2}^s \cdots \sigma_{p}^s ~|~0 \leq s \leq{M-1},  \sigma_{i}^s \in \langle \sigma_{i}\rangle _{j}=1 \}\\
  =& \{\sigma_{1}^s \cdots \sigma_{i-1}^s \sigma_{i+1}^s \cdots \sigma_{p}^{s} ~|~ 0 \leq s \leq {M-1}, m_{i}| s \}.  
\end{align*}
In this case, it is easy to see that $|G_j|=\frac{M}{m_i}$. 
Now for $j\in \mathcal{O}_{i}$ we have
 \begin{align*}
 [\lambda_q, 1_{G_j}]=&\frac{1}{|G_j|}\sum_{g\in G_j}\lambda _{q}(g)  \hspace{10mm} \\
 =&\frac{1}{|G_j|}\sum_{g\in G_j}\lambda_q ((\sigma_{1} \sigma_{2}\cdots  \sigma_{p})^s) \\
  =&\frac{m_i}{M}\sum_{s=0,\  m_{i}\mid s}^{M-1} \exp\left( \dfrac{2 \pi iqs}{M}\right) 
\end{align*}
For any $1\leq i\leq p$ we choose $j_i \in \mathcal{O}_{i}$. Then $\mathcal{D}=\{ j_1,\cdots,j_p \}$. Therefore
\begin{eqnarray}
 \dim V^{\lambda _q}(G) &=& (\dim V) \lambda _q (1)\sum_{j\in \bar{\mathcal{D}}}[\lambda_q, 1_{G_{j}}]\nonumber\\
 &=& (\dim V) \sum_{j\in \mathcal{D}}[\lambda_q, 1_{G_{j}}]\nonumber\\
 &=& (\dim V)\sum_{i=1}^{p} [\lambda_q, 1_{G_{j_i}}]\nonumber\\
 &=& (\dim V)\sum_{i=1}^{p}\frac{m_i}{M}
   \sum_{s=0, m_{i}\mid s}^{M-1} \exp\left( \dfrac{2 \pi iqs}{M}\right)\nonumber
\end{eqnarray}
Now letting $s_{i}^{\prime}=s/m_{i}$ and $m_{i}^{\prime}=M/m_{i}$, we obtain
\begin{eqnarray}
 \dim V^{\lambda _q}(G)
 &=& (\dim V)\sum_{i=1}^{p}\frac{1}{m_{i}^{\prime}}
   \sum_{s_{i}^{\prime}=0}^{m_{i}^{\prime}-1} \exp\left( \dfrac{2 \pi iqs_{i}^{\prime}}{m_{i}^{\prime}}\right)\nonumber\\
   &=& (\dim V)\sum_{i=1}^{p}\frac{1}{m_{i}^{\prime}}
 \sum_{d\mid m_{i}^{\prime}}~~~  \sum_{(s_{i}^{\prime},m_{i}^{\prime})=d} \exp\left( \dfrac{2 \pi iqs_{i}^{\prime}}{m_{i}^{\prime}}\right)\nonumber\\
 &=& (\dim V)\sum_{i=1}^{p}\frac{1}{m_{i}^{\prime}}
 \sum_{d\mid m_{i}^{\prime}} \mathit{c}_{m_{i}^{\prime}/d}(q). \nonumber
\end{eqnarray}
\end{proof}
Our the other goal is to obtain the dimension of $V^{\lambda}(G)$,  when $G\leq S_m$ has the structure 
$$
G=<\sigma_{1}><\sigma_{2}> \cdots <\sigma_{k}>,
$$
with $\sigma_{\ell}$ $(1 \leq \ell \leq k)$ representing disjoint cycles in $S_{m}$ of specific orders, denoted as $m_{1},\ldots,m_{k}$, respectively. The irreducible characters of $G$ are exclusively linear and can be expressed as products of the irreducible characters of the cyclic groups $\langle\sigma_{\ell}\rangle$, where $\ell=1, \ldots, k$:
$$
\lambda_{(q_1,\ldots,q_k)} = \lambda_{q_1}\cdots \lambda_{q_k},~ q_{\ell}=0,\ldots,m_{\ell}-1,
$$
and the character $\lambda_{q_\ell}$ is defined as 
$\lambda_{q_\ell}(\sigma_{\ell}^{j_{\ell}})=\exp \left( \dfrac{2 \pi i{q_{\ell}}{j_{\ell}}}{m_{\ell}} \right)$.
 \begin{theorem}
Suppose $G=<\sigma_{1}><\sigma_{2}> \cdots <\sigma_{k}>, k\geq 2$, where $\sigma_{\ell}, \ell=1,\ldots,k$, are disjoint cycles in $S_{m}$ of orders $m_1,\cdots,m_{k}$. Let $\lambda = \lambda_{(q_1,\ldots ,q_k)}$, where $0\leq q_\ell\leq m_\ell-1$, $1\leq \ell\leq k$. Then
$$
 \dim V^{\lambda}(G)=(\dim V)  \left[(m-\sum_{\ell=1}^{k}m_{\ell})[\lambda , 1_{G}]+ \prod_{l \neq j}\frac{1}{m_\ell}\sum_{d\mid m_\ell} \mathit{c}_{m_{\ell}/d}(q_\ell)\right].
$$
In the case $k=1$, $\dim V^{\lambda}(G)=(\dim V)\left[(m-m_1)[\lambda , 1_{G}]+1\right]$.
\end{theorem}
\begin{proof}
For any $j \in \mathcal{D}$, if $\sigma_{i}(j)=j$ for all $1\leq i\leq k$, then $G_{j}=G$ and so $[\lambda , 1_{G_
j}]=[\lambda , 1_{G}]$.
Otherwise there exists $1\leq i
\leq k$, such that $\sigma_{i}(j) \neq j $. Then $G_{j} = \prod_{i \neq j}<\sigma_i>$. Thus we have
 \begin{align*}
 [\lambda, 1_{G_j}]=&\frac{1}{|G_j|}\sum_{g\in G_j}\lambda (g)  \hspace{10mm}  \\
 =&\prod_{\ell \neq j}\frac{1}{m_\ell}\sum_{j_\ell=0}^{m_\ell-1}\lambda_{q_\ell}(\sigma_{\ell}^{j_{\ell}})  \\
  =&\prod_{\ell \neq j}\frac{1}{m_\ell}\sum_{j_\ell=0}^{m_\ell-1} \exp \left( \dfrac{2 \pi i{q_{\ell}}{j_{\ell}}}{m_{\ell}} \right)\\
=&\prod_{\ell \neq j}\frac{1}{m_\ell}\sum_{d\mid m_\ell}\sum_{d=(j_\ell, m_\ell)} \exp \left( \dfrac{2 \pi i{q_{\ell}}{j_{\ell}}}{m_{\ell}} \right)\\
=&\prod_{\ell \neq j}\frac{1}{m_\ell}\sum_{d\mid m_\ell} \mathit{c}_{m_\ell/d}(q_\ell).
  \end{align*}
Now, by letting $\mathcal{A} = \{ j ~|~  \sigma_{i}(j) = j~\text{for all}~1\leq i\leq k\}$, we have

\begin{eqnarray}
 \dim V^{\lambda}(G) &=&(\dim V) \lambda (1)\sum_{j\in {\mathcal{D}}}[\lambda, 1_{G_j}] \nonumber\\
  &=&(\dim V) \left[\sum_{j\in \mathcal{A}}[\lambda , 1_{G}]+ \sum_{j=1}^{k}[ \lambda, 1_{G_j}]\right],\nonumber\\
 &=&(\dim V)  \left[(m-\sum_{\ell=1}^{k}m_{\ell})[\lambda , 1_{G}]+ \prod_{\ell \neq j}\frac{1}{m_\ell}\sum_{d\mid m_\ell} \mathit{c}_{m_\ell/d}(q_\ell)\right],\nonumber
\end{eqnarray}
so the result holds.
\end{proof}
\section{\bf Cartesian symmetry classes associated with dihedral group}
In this section, we first obtain the dimensions of Cartesian symmetry classes associated with
the irreducible characters of the dihedral group $D_{2m}$. \\

The subgroup $D_{2m}$ of $S_{m}~(m\geq 3)$ generated by the elements
$$\ r=(1\ 2\ \cdots\ m)\ \ \text {and}\ \  s=\begin{pmatrix}
  _{1} & _{2} & _{3} & _{\cdots} & _{m-1} & _{m} \\
  _{1} & _{m} & _{m-1} & _{\cdots} & _{3} & _{2}
\end{pmatrix}
$$
is the {\em dihedral group of degree $m$} (see \cite[page 50]{Hungerford}). The generators $r$
and $s$ satisfy
$$
r^{m}=1=s^{2}\ \text{and}\ s^{-1}r s=r^{-1}.
$$
 In particular,
$$
D_{2m}=\{r^{k},\ sr^{k}|\ 0\leq k<m\}.
$$
If $m$ is even, i.e., $m = 2k$~($k \geq 2$), then $D_{2m}$ has $k + 3$ conjugacy classes.
If $m$ is odd, i.e., $m = 2k + 1$~($k \geq 1$), then $D_{2m}$ has $k + 2$ conjugacy classes.\\

For each integer $h$ with $0<h<m/2$, $D_{2m}$ has an irreducible
character $\psi_{h}$ of degree $2$ given by
$$\psi_{h}(r^{k})=2\cos \frac{2kh\pi}{m},\ \
\psi_{h}(sr^{k} )=0,\ \ \  0 \leq k < m.
$$
The other characters of $D_{2m}$ are of degree $1$, namely
$\lambda_{j}$. The character table of $D_{2m}$ is shown in Table 1 (see \cite[page 182]{James}).

Clearly, the action of $G=D_{2m}$ on $\mathbf{I}_m=\{ 1,2,\cdots, m \}$ is transitive, so $\mathcal{D}=\{1\}$. On the other hand, $G_1=\{1, s\}$. Thus for any $\lambda\in \text{Irr}(G)$, we have 
\begin{eqnarray}\label{dimD_2m}
\dim V^{\lambda} (G)=(\dim V)~\lambda (1)~[\lambda, 1_{G_1}]=(\dim V)~\frac{\lambda (1)}{2}~[\lambda (1)+\lambda (s)].
\end{eqnarray}
Now, by using Table 1 and Equation (\ref{dimD_2m}), we have the following result.
\begin{theorem} 
 Let $G=D_{2m}~(m \geq 3)$.  Assume $n = \dim V \geq 2$. \\
 
{\bf Case (i)} If $m$ is even, then 

\begin{itemize}

\item[(a)]
$\dim V^{\lambda_{1}} (G)=\dim V^{\lambda_{3}} (G)=n$, \\

\item[(b)]
$\dim V^{\lambda_{2}} (G)=\dim V^{\lambda_{4}} (G)=0$, \\

\item[(c)]
$\dim V^{\psi_{h}} (G)=2n~(0< h< \frac{m}{2})$.

 \end{itemize}
 {\bf Case (ii)} If $m$ is odd, then 
\begin{itemize}

\item[(a)]
$\dim V^{\lambda_{1}} (G)=n$, \\

\item[(b)]
$\dim V^{\lambda_{2}} (G)=0$, \\

\item[(c)]
$\dim V^{\psi_{h}} (G)=2n~(0< h< \frac{m}{2})$.
 \end{itemize}
\end{theorem}

In the following theorem we give a necessary and sufficient condition for the existence an $O$
-basis for Cartesian symmetry classes  $V^{\psi_h}(G)~(0<h<\frac{m}{2})$. 
\begin{theorem}\label{o-basis D_2m}
Let $G=D_{2m}~(m\geq 3)$ and $\psi=\psi_{h}~(0<h<\frac{m}{2})$. Then $V^{\psi}(G)$ has $O$-basis if and only if $m\equiv 0~(\text{mod}~ 4h_2)$,
where $h_{2}$ is the $2$-part of $h$. 
\end{theorem}
\begin{proof}
Since $G$ acts on $\mathbf{I}_m$ transitively, so $\mathcal{D}=\{1\}$. Clearly, $G_{1}=\{1, s\}$. So $[\psi, 1_{G_1}]=\frac{1}{2}[\psi (1)+\psi (s)]=1\neq 0$. This implies that $1\in \bar{\mathcal{D}}$. But $V^{\psi}(G)=\sum_{1\leq i\leq n}^{\bot}
 V^{\psi}_{i1}$. Notice that that $\dim V^{\psi}_{i1}=2$ and the set $\{f^{\psi}_{ir(1)}, f^{\psi}_{is(1)}\}$ is a basis for it.
By using Theorem \ref{p1}, we have
$$
\langle f^{\psi}_{ir(1)}, f^{\psi}_{is(1)}\rangle=\frac{\psi (1)}{|G|}\sum_{g\in G_{1}}\psi (gsr^{-1})=\frac{1}{m}[\psi (sr^{-1})+\psi (r^{-1})]=
\frac{2}{m}\cos~\frac{2\pi h}{m}.
$$
It is observe that $\langle f^{\psi}_{ir(1)}, f^{\psi}_{is(1)}\rangle =0$ if and only if $\cos~\frac{2\pi h}{m}=0$, or equivalently $m\equiv 0~(\text{mod}~ 4h_2)$,
$h_{2}$ is the $2$-part of $h$. This completes the proof.
\end{proof}
\begin{corollary}
Let $G=D_{2m}$ and assume $\dim V\geq 2$. Then $\times_{1}^{m}V$ has an $O$-basis if and only if $m$ is a power of $2$.
\end{corollary}
\begin{proof}
Suppose $m_{2}$ is the $2$-part of $m$. Assume that $m_2<m$. Then $0<m_2<\frac{m}{2}$ and 
$m\not\equiv 0~(\text{mod}~ 4h_2)$. Hence, if $\psi=\psi_{m_2}$, then Theorem \ref{o-basis D_2m} implies that $V^{\psi}(G)$ has no $O$-basis, and so $\times_{1}^{m}V$.\\
Conversely, assume $m$ is a power of $2$. If $0<h<\frac{m}{2}$, then $h_2\leq \frac{m}{4}$ and $m\equiv 0~(\text{mod}~ 4h_2)$, where $h_2$ is the $2$-part of $h$. Now Theorem \ref{o-basis D_2m} implies that $V^{\psi_{h}}(G)$ has an $O$-basis, and so $\times_{1}^{m}V$.
\end{proof}
\begin{remark}
Let $G=D_{8}\leq S_4$ . Then $G$ has only one non-linear irreducible character $\psi (r^k)=2\cos \frac{k\pi}{2},~\psi (sr^k)=0,~0\leq k<4$. By Theorem \ref{o-basis D_2m}, $V^{\psi}(G)$ has an $O$-basis. Again, by Theorem \ref{o-basis D_2m}, we have
$$
 \parallel f^{\psi}_{ij}\parallel ^{2}=\frac{\psi (1)}{[G:G_1]}[\psi, 1_{G_1}]=\frac{1}{4}[\psi (1)+\psi (s)]=\frac{1}{2}.
$$
So $\parallel f^{\psi}_{ij}\parallel =\frac{\sqrt{2}}{2}$. This examle show that $\frac{\sqrt{2}}{2}$ is the best lower bound for $\parallel f^{\lambda}_{ij}\parallel$ in Theorem \ref{o-basis}. 
\end{remark}
Now we give a counter-example that the converse of Theorem \ref{o-basis} is not true.
\begin{example}
Let $q>3$ be a prime and 
$$G=D_{2q}=
<r^{q}=1=s^{2}\ \text{and}\ s^{-1}r s=r^{-1}>.
$$
Consider $G$ as a subgroup of $S_{q}$. Then the non-linear irreducible characters of $D_{2q}$ are $\psi_h$, where $0<h<q/2$ and $\psi_{h} (r^k)=2\cos \frac{2k\pi h}{q},~\psi_{h} (sr^k)=0,~0\leq k< q$.
Let $\psi=\psi_h,~ 0<h<\frac{q}{2}$. By Theorem \ref{o-basis D_2m}, $V^{\psi}(G)$ has no $O$-basis. But we have 
$$
 \parallel f^{\psi}_{i1}\parallel ^{2}=\frac{\psi (1)}{[G:G_1]}[\psi, 1_{G_1}]=\frac{1}{q}[\psi (1)+\psi (s)]=\frac{2}{q}<\frac{1}{2}.
$$
\end{example}


\end{document}